\documentclass[12pt,psamsfonts]{amsart}
\usepackage{euscript,amsmath,amsthm,amsfonts,amssymb}
\usepackage{graphicx}
\usepackage[all,knot]{xy}
\xyoption{arc}
\usepackage{epsfig}
\usepackage[dvips]{pstricks}

\newcommand{\e}{{\epsilon}}
\newcommand{\I}{\mathcal{I}}

\newcommand{\C}{\mathbb C}

\newcommand{\ot}{\otimes}

\numberwithin{equation}{section}

\newtheorem{theorem}{Theorem}[section]

\newtheorem{cor}[theorem]{Corollary}

\newtheorem{lemma}[theorem]{Lemma}

\newtheorem{prop}[theorem]{Proposition}
\theoremstyle{definition}

\newtheorem{defn}[theorem]{Definition}

\begin{document}
\title[]
{Generalized Yang-Baxter equations and Braiding Quantum Gates}

\author{Rebecca S. Chen}
\email{rchen@parktudor.org}
\address{Park Tudor School\\
7200 North College Avenue\\
    Indianapolis, IN 46240\\
    U.S.A.}

\thanks{We thank Prof.\ Rowell for suggesting the problem and for sending me his Maple program and his solution.  This project would not exist without his guidance and generous help.}

\begin{abstract}

Solutions to the Yang-Baxter equation - an important equation in mathematics and physics - and their afforded braid group representations have applications in fields such as knot theory, statistical mechanics, and, most recently, quantum information science.  In particular, unitary representations of the braid group are desired because they generate braiding quantum gates.  These are actively studied in the ongoing research into topological quantum computing.  A generalized Yang-Baxter equation was proposed a few years ago by Eric Rowell et al.  By finding solutions to the generalized Yang-Baxter equation, we obtain new unitary braid group representations.  Our representations give rise to braiding quantum gates and thus have the potential to aid in the construction of useful quantum computers.

\end{abstract}

\maketitle

\section{Introduction}

In 1944, L.~Onsager published his now-famous solution of the Ising model in statistical mechanics.  He employed in his solution a clever relation called a star-triangle transformation \cite{PA}, better known today as the Yang-Baxter equation.  Although Onsager never fully realized the significance of the star-triangle transformation, R.~J.~Baxter did, and he demonstrated its importance in the exact solutions of other statistical mechanical models.  C.~N.~Yang, working independently to solve one-dimensional many-body problems with delta-function interactions, came across the same transformation.  The term Yang-Baxter equation was coined by L.~Faddeev in the late 1970s in honor of these two physicists.  Now the Yang-Baxter equation is an important equation in mathematics and physics, connected to a variety of fields such as statistical mechanics, quantum field theory, quantum topology, quantum groups, and, most recently, quantum information science.

The Yang-Baxter equation in dimension $d$ (some would say in dimension $d^2$) is a matrix equation for an invertible complex matrix $R=(R^{kl}_{ij}),i,j,k,l=1,2,\cdots, d$.  The easiest way to write down the Yang-Baxter equation is to use the language of linear operators between vector spaces.  Let $V$ be a $d$-dimensional complex vector space with a chosen basis $\{e_i\},i=1,2,\cdots, d$.
The matrix $R$ defines an invertible operator $R: V\otimes V \rightarrow V\otimes V$ by $R(e_i\otimes e_j)=\sum_{k,l=1}^d R^{kl}_{ij}e_k\otimes e_l, i,j=1,...,d$,
where $\{e_a\otimes e_b\}, a,b=1,..,d$, is a basis of $V\ot V$.  By abusing notation, we have denoted the operator associated to $R$ also by $R$.  Letting $I_V$ be the identity operator on $V$, we form two operators $R\ot I_V$ and $I_V\ot R$ from $V\otimes V \otimes V$ to itself.  Then the Yang-Baxter equation, written as an equation for linear operators, is the following:

\vspace{.1in}
\begin{enumerate}\label{YBE:operatorform}
 \item[(YBE)] $(R\ot I_V)(I_V\ot R)(R\ot I_V)=(I_V\ot R)(R\ot I_V)(I_V\ot R)$.
\end{enumerate}
\vspace{.1in}

When we use the basis $\{e_a\ot e_b \ot e_c\}, a,b,c=1,...,d$, the above equation for linear operators becomes an equation for the matrix $R$.  This matrix equation for $R$ is usually called the Yang-Baxter equation, and a solution is called an $R$-matrix.  The two equations are in one-to-one correspondence when bases of the involved vector spaces are fixed, so we will speak of both as the Yang-Baxter equation with the understanding that some basis has been chosen for the operator equation.  Strictly speaking, the Yang-Baxter equation here is the \emph{braided} version of the constant quantum Yang-Baxter equation, but we refer to it simply as the Yang-Baxter equation for brevity.

The Yang-Baxter equation consists of a set of polynomial equations for the entries $\{R^{kl}_{ij}\}$ of the matrix $R$.  Explicitly, for any choice of two ordered triples $(x,y,z)$ and $(u,v,w)$ from $\{1,2,\cdots, d\}$, we have

\vspace{.1in}
\begin{enumerate}\label{YBE:matrixform}
 \item[(YBE)] $\sum_{a,b,c=1}^d R^{ab}_{uv}R^{cz}_{bw}R^{xy}_{ac}=\sum_{m,n,p=1}^d R^{np}_{vw}R^{xm}_{un}R^{yz}_{mp}$.
\end{enumerate}
\vspace{.1in}

One application of unitary $R$-matrices is to quantum information science.  A unitary $R$-matrix leads to a unitary representation of the braid group, and the resulting unitary matrices associated to braids can be used to process quantum information \cite{Nayak}.  Inspired by such an application, E.~Rowell et al proposed a generalized version of the Yang-Baxter equation in \cite{RZWG}.  Unitary solutions to the generalized Yang-Baxter equation sometimes also afford braid group representations and therefore can be used for quantum information processing.  An $8\times 8$ solution to a generalized Yang-Baxter equation is used to generate the Greenberger-Horne-Zeillinger states \cite{RZWG}.

Solutions to the Yang-Baxter equation are difficult to find.  This can be seen by counting the number of variables and the number of equations.  If the vector space $V$ is of dimension $d$, then $R$ has $d^4$ entries, so there are $d^4$ unknowns.  The Yang-Baxter equation consists of $d^6$ cubic polynomial equations for the $d^4$ variables $\{R^{kl}_{ij}\}$.  For $d=1$, the equation is automatically satisfied by any nonzero complex number.  For $d=2$, the Yang-Baxter equation consists of $64$ cubic homogeneous polynomial equations for $16$ complex unknowns.  This is the only case for which the Yang-Baxter equation has been solved completely, albeit with the help of a computer.  The general solutions are found in \cite{H}, and based on this work, the unitary ones are classified in \cite{Dye}.  The case of $d=3$ consists of  $729$ cubic polynomial equations for $81$ unknowns.  No wonder it has never been completely solved!

Furthermore, new solutions to generalized Yang-Baxter equations are even more challenging to discover.  So far, only two essentially new solutions have been found: the $8\times 8$ unitary solution given in \cite{RZWG} and a unitary solution found by Prof.\ Rowell based on an unpublished work of Prof.\ Goldschmidt and Prof.\ Jones (cf.~\cite{GHR, R}).  We will refer to the latter as the Rowell solution.  In this paper, we generalize the Rowell solution to three families of unitary solutions, each of which is parameterized by the upper semicircle in the complex plane.  Our solutions lead to new unitary braid group representations that generate braiding quantum gates.

Our solutions were found with the help of the computer package Maple.  Modifying a short program written by Prof.\ Rowell, we found the solutions up to an error of order $10^{-11}$.  Then we proved algebraically that they are indeed solutions.

The contents of the paper are as follows.  In Section \ref{gYBE}, we review the generalized Yang-Baxter equations.  Section \ref{solns} contains our main result: certain solutions to a generalized Yang-Baxter equation.  In Section \ref{quantumgates}, we show that
all of our solutions in Section \ref{solns} give rise to unitary representations of the braid group.  In Section \ref{questions}, we list some open problems and directions for future research.

\section{Generalized Yang-Baxter Equations}\label{gYBE}

The Yang-Baxter equation (YBE) is indexed by a single natural number $d$, the dimension of the vector space $V$.
The \emph{generalized Yang-Baxter equation} (gYBE) proposed in \cite{RZWG} is indexed by two more natural numbers, $m$ and $l$.  For convenience, we will continue to use operator language while keeping in mind that, by the choice of a basis of the vector space, the gYBE is a matrix equation.

\begin{defn}
Let $V$ be a complex vector space of dimension $d$.  The $(d,m,l)$-gYBE is an equation for an invertible operator $R: V^{\ot m}\rightarrow V^{\ot m}$ such that
\vspace{.1in}
    \begin{enumerate}\label{gYBE:operatorform}
    \item[(gYBE)] $(R\ot I_V^{\ot l}) (I_V^{\ot l}\ot R) (R\ot I_V^{\ot l})=(I_V^{\ot l}\ot R) (R\ot I_V^{\ot l}) (I_V^{\ot l}\ot R)$,
    \end{enumerate}
\vspace{.1in}
where $d$, $m$, and $l$ are natural numbers and $I_V^{\ot l}$ is the identity operator on $V^{\ot l}$.
Any matrix solution to the $(d,m,l)$-gYBE is called a $(d,m,l)$-$R$-matrix.
\end{defn}

Note that if $m=2$ and $l=1$, the gYBE reduces to the usual YBE.  Generally, however, the gYBE is harder to solve than the YBE, with the exception of a couple of basic cases.  When $d=1$, $R$ is just a scalar, so any non-zero complex number is a solution.  When $m=l=1$, the gYBE becomes the equation $R^2\otimes R=R\otimes R^2$, where $R$ is an invertible operator on $V$.  $R^2\otimes R=R\otimes R^2$ is the same as $R\otimes I=I\otimes R$ because $R$ is invertible.  It follows that $R=\lambda I_V$ for some nonzero scalar $\lambda$, if $m=l=1$.

For the application to quantum information science, we will focus on $d=2$ because, when $d=2$, $V$ is isomorphic to $\mathbb{C}^2$, the so-called qubit state space.  The $(2,2,1)$-gYBE is the YBE in dimension $2$, so the first non-trivial gYBE for qubits is the $(2,3,1)$-gYBE.

\begin{defn}
The $(2,3,1)$-gYBE for  $R: (\mathbb{C}^2)^{\ot 3}\rightarrow (\mathbb{C}^2)^{\ot 3}$ is $$R_1R_2R_1=R_2R_1R_2,$$ where $R_1=R\otimes I_2$ and $R_2=I_2\otimes R$ act on $(\mathbb{C}^2)^{\ot 4}$.
\end{defn}

There are several different conventions regarding the tensor product of matrices.  We will use the so-called Kronecker product: for two matrices $A=(a_{ij})_{m\times n}$ and $B=(b_{kl})_{p\times q}$, $A\otimes B$ is the $(mp\times nq)$-matrix obtained by replacing each entry $a_{ij}$ of $A$ by the block $a_{ij}B$.  We will use $I_n$ to denote the $n\times n$ identity matrix.  When no confusion would result, we will simply write $I_n$ as $I$.

For two matrices $X$ and $Y$,  $X\oplus Y$ denotes the block diagonal matrix
$\begin{pmatrix}
X&0\\0&Y
\end{pmatrix}.$

\begin{prop}\label{gYBE:double}
Suppose $X$ is a $4\times 4$ matrix.  Then $X\oplus X$ is a solution to the $(2,3,1)$-gYBE if and only if $X$ is a solution to the YBE in dimension $2$.
\end{prop}

A proof is contained in the proof of Lemma \ref{gYBE:Block}.

As mentioned in the introduction, there currently exist only two essentially new solutions for any gYBE other than the $(d,1,1)$-gYBE.  Up to simple changes described in Prop.\ \ref{YBE:gauge} below, the first solution is a $(2,3,2)$-$R$-matrix in Section $4.2$ of \cite{RZWG}:
$$R_X=\frac{1}{\sqrt{2}}
\begin{pmatrix}
1&0&0&0 &0 & 0&0& 1\\
0&1&0&0 &0 & 0&1& 0\\
0&0&1&0 &0 & 1&0& 0\\
0&0&0&1 &1 & 0&0& 0\\
0&0&0&-1 &1 & 0&0& 0\\
0&0&-1&0 &0 & 1&0& 0\\
0&-1&0&0 &0 & 0&1& 0\\
-1&0&0&0&0&0&0&1
\end{pmatrix}.$$

We will refer to this solution as the $X$-shape solution as the non-zero entries form the shape of an $X$.

The second solution is the Rowell solution, a $(2,3,1)$-$R$-matrix.  Let $\zeta=e^{2\pi i/8}$.  Then,
$$R_{\zeta}=\frac{1}{\sqrt{2}}
\begin{pmatrix}
\zeta^{-1}&0&-\zeta^{-1}&0\\0 &\zeta &0&\zeta\\\zeta&0&\zeta&0\\0&-\zeta^{-1}&0&\zeta^{-1}
\end{pmatrix}\oplus \frac{1}{\sqrt{2}}
\begin{pmatrix}
\zeta&0&\zeta&0\\0 &\zeta^{-1} &0&-\zeta^{-1}\\-\zeta^{-1}&0&\zeta^{-1}&0\\0&\zeta&0&\zeta
\end{pmatrix}.$$

Any solution to the gYBE, just as for the YBE, leads to many more solutions by the following proposition.

\begin{prop}\label{YBE:gauge}
If $R$ is a solution to the $(d,m,l)$-gYBE, then
    \begin{enumerate}
    \item $\lambda R$ is also a solution for any nonzero scalar $\lambda$.
    \item $R^{-1}$ is also a solution.
    \item ${(Q^{-1})}^{\ot m}\cdot R \cdot Q^{\ot m}$ is also a solution, where $Q$ is an invertible $d\times d$ matrix.
    \end{enumerate}
\end{prop}

The proof for the YBE in \cite{Dye} works for the gYBE without any change.  The proof of the third case is especially interesting.

One of the reasons for interest in the YBE is that any solution leads to a matrix representation of the braid group.  This is not generally true for solutions to the gYBE, a point that we discuss in Section \ref{quantumgates}, though the $X$-shape solution $R_X$ and the Rowell solution $R_{\zeta}$ do lead to braid group representations.

\section{Solutions to the (2,3,1)-Generalized Yang-Baxter Equation}\label{solns}

In this section, we focus on the gYBE for $(d,m,l)=(2,3,1)$ and generalize the only known non-trivial $(2,3,1)$-$R$-matrix: the Rowell solution $R_{\zeta}$.
Our strategy is to search for unitary $(2,3,1)$-$R$-matrices with the non-zero entries in the same positions as in $R_{\zeta}$.  We found that the Rowell solution lies in a one-parameter family and discovered two more families of new unitary solutions.  In Section \ref{quantumgates}, we show that these solutions give rise to unitary braid group representations.

The Rowell solution can be rewritten as
$$R_{\zeta}=\frac{{\zeta^{-1}}}{\sqrt{2}}
\begin{pmatrix}
1&0&-1&0\\0 & i&0& i\\i&0&i&0\\0&-1&0& 1
\end{pmatrix}\oplus \frac{{\zeta^{-1}}}{\sqrt{2}}
\begin{pmatrix}
i&0&i&0\\0 & 1 &0&-1\\-1&0&1&0\\0& i&0& i
\end{pmatrix}.$$

By Prop.\ \ref{YBE:gauge}, the Rowell solution multiplied by $\zeta$ is also a solution.

\begin{defn}
Two $(2,3,1)$-$R$-matrices $R$ and $R'$ are
    \begin{enumerate}
    \item locally conjugate if there is an invertible $2\times 2$ matrix $Q$ such that $R=(Q^{-1}\otimes Q^{-1}\otimes Q^{-1})R'(Q\otimes Q\otimes Q)$.
    \item equivalent if $R$ and $R'$ are related by a sequence of applications of the three relations in Prop.\ \ref{YBE:gauge} (i.e., $R$ can obtained from $R'$ by multiplying a non-zero scalar, by taking the inverse, by a local conjugation, or by an arbitrary combination of these three operations).
    \end{enumerate}
\end{defn}

\begin{defn}
A $4\times 4$ matrix $M$ is
    \begin{enumerate}
    \item $(2\times 2)$-diagonal if $M=
        \begin{pmatrix}
        M_{11}&M_{12}\\
        M_{21}&M_{22}
        \end{pmatrix}$ such that $M_{ij}, i,j=1,2,$ are all diagonal $2\times 2$ matrices.
    \item $(2\times 2)$-diagonally unitary if it is $(2\times 2)$-diagonal and unitary, and $M=\frac{1}{\sqrt{2}}
        \begin{pmatrix}
        M_1&M_2\\
        M_3&M_4
        \end{pmatrix}$ such that each block $M_i, i=1,2,3,4,$ is unitary.
    \end{enumerate}
\end{defn}

Our main result is the following:

\begin{theorem}\label{Thm:Rowell}
If an $8\times 8$ unitary matrix solution $R$ to the $(2,3,1)$-gYBE is of the form $R=X\oplus Y$, where the $4\times 4$ matrix $X$ is $(2\times 2)$-diagonally unitary, then
    \begin{enumerate}
    \item $R$ is equivalent to an $R(\theta)$ in one of the following three families for some $0\leq \theta \leq \pi$:
        \begin{enumerate}
        \item $$R(\theta)=\frac{1}{\sqrt{2}}
            \begin{pmatrix}
            1&0&1 &0\\0 & i&0& e^{i\theta}\\-i&0&i&0\\0&-i e^{-i\theta}&0& 1
            \end{pmatrix}\oplus \frac{1}{\sqrt{2}}
            \begin{pmatrix}
            i&0&e^{i\theta}&0\\0 & 1 &0&-e^{2i\theta}\\-ie^{-i\theta}&0&1&0\\0& ie^{-2i\theta}&0& i
            \end{pmatrix},$$
        \item $$R(\theta)=\frac{1}{\sqrt{2}}
            \begin{pmatrix}
            1&0&1 &0\\0 & i&0& e^{i\theta}\\-1&0&1&0\\0&e^{-i\theta}&0& i
            \end{pmatrix}\oplus \frac{1}{\sqrt{2}}
            \begin{pmatrix}
            i&0&e^{i\theta}&0\\0 & 1 &0&-e^{2i\theta}\\e^{-i\theta}&0&i&0\\0& e^{-2i\theta}&0& 1
            \end{pmatrix},$$
        \item $$R(\theta)=\frac{1}{\sqrt{2}}
            \begin{pmatrix}
            1&0&1 &0\\0 & 1&0& e^{i\theta}\\-1&0&1&0\\0&-e^{-i\theta}&0& 1
            \end{pmatrix}\oplus \frac{1}{\sqrt{2}}
            \begin{pmatrix}
            1&0&-e^{i\theta}&0\\0 & 1 &0&-e^{2i\theta}\\e^{-i\theta}&0&1&0\\0& e^{-2i\theta}&0& 1
            \end{pmatrix}.$$
        \end{enumerate}
    \item Any two different $(2,3,1)$-$R$-matrices in the three families above are not equivalent to each other.
    \item For each $(2,3,1)$-$R$-matrix above, $X$ is different from $Y$ except when $\theta=\pi$ in the third family.  Therefore, neither $X$ nor $Y$ is a solution to the YBE unless for $\theta=\pi$ in the third family.
    \end{enumerate}
\end{theorem}

For $0\leq \theta \leq \pi$, $e^{i\theta}$ is the upper semicircle in the complex plane.  Therefore, solutions in each family correspond to points in the upper semicircle.  For example, the Rowell solution corresponds to the point $i$ in the complex plane; i.e., the Rowell solution is equivalent to $R(\pi/2)$ in the first family.  When $\theta=\pi$ in the third family, $$X(\pi)=Y(\pi)=\frac{1}{\sqrt{2}}
\begin{pmatrix}
1&0&1 &0\\0 & 1&0& -1\\-1&0&1&0\\0& 1&0& 1
\end{pmatrix}$$
is a YBE solution that is locally conjugate to the $R$-matrix corresponding to the Bell states \cite{Dye} \cite{KL}.

\subsection{Proof of Main Theorem \ref{Thm:Rowell}}

We start the proof of Thm.\ \ref{Thm:Rowell} by presenting several lemmas.

In this subsection, we use $R,X,Y,A,B,C,D,Y_i,i=1,2,3,4,$ and $\alpha,\beta,\omega,\gamma,\delta$ to denote the following matrices and matrix entries. The $8\times 8$ matrix $R$ is $X\oplus Y$ for some $4\times 4$ matrices $X$ and $Y$.  The matrices $X$ and $Y$ are written as $X=\frac{1}{\sqrt{2}}
\begin{pmatrix}
A&B\\C&D
\end{pmatrix}$ and $Y=\frac{1}{\sqrt{2}}
\begin{pmatrix}
Y_1&Y_2\\Y_3&Y_4
\end{pmatrix}$.  Let $A=\begin{pmatrix}
1&0\\0&\omega
\end{pmatrix},
B=\begin{pmatrix}
\alpha &0\\0&\beta
\end{pmatrix},
D=\begin{pmatrix}
\gamma &0\\0&\delta
\end{pmatrix}.$
We also denote the identity matrix $I_2$ simply by $I$.

\begin{lemma}\label{Lemma:unitary}
If $X$ is $(2\times 2)$-diagonally unitary, then $C=-DB^{\dagger}A=\begin{pmatrix}-\gamma \bar{\alpha} &0\\0&-\delta \omega \bar{\beta}\end{pmatrix}.$
\end{lemma}

\begin{proof}
$XX^{\dagger}=I_4$ implies that $CA^{\dagger}+DB^{\dagger}=0$. Thus, $C=-DB^{\dagger}(A^{\dagger})^{-1}$.  If $A$ is unitary, then $(A^{\dagger})^{-1}=A$. A simple computation yields $C$ in terms of $\alpha, \beta, \omega, \gamma, \delta.$

\end{proof}

\begin{lemma}\label{gYBE:Block}
An $8\times 8$ matrix $R=X\oplus Y$ is a solution to the $(2,3,1)$-gYBE if and only if the following equations are satisfied:
    \begin{enumerate}
    \item $(A\otimes I)X (A\otimes I)+(B\otimes I) Y(C\otimes I)=\sqrt{2}\cdot X(A\otimes I)X$
    \item $(A\otimes I)X (B\otimes I)+(B\otimes I) Y(D\otimes I)=\sqrt{2}\cdot X(B\otimes I)Y$
    \item $(C\otimes I)X (A\otimes I)+(D\otimes I) Y(C\otimes I)=\sqrt{2}\cdot Y(C\otimes I)X$
    \item $(C\otimes I)X (B\otimes I)+(D\otimes I) Y(D\otimes I)=\sqrt{2}\cdot Y(D\otimes I)Y$
    \item $(Y_1\otimes I)X (Y_1\otimes I)+(Y_2\otimes I) Y(Y_3\otimes I)=\sqrt{2}\cdot X(Y_1\otimes I)X$
    \item $(Y_1\otimes I)X (Y_2\otimes I)+(Y_2\otimes I) Y(Y_4\otimes I)=\sqrt{2}\cdot X(Y_2\otimes I)Y$
    \item $(Y_3\otimes I)X (Y_1\otimes I)+(Y_4\otimes I) Y(Y_3\otimes I)=\sqrt{2}\cdot Y(Y_3\otimes I)X$
    \item $(Y_3\otimes I)X (Y_2\otimes I)+(Y_4\otimes I) Y(Y_4\otimes I)=\sqrt{2}\cdot Y(Y_4\otimes I)Y$
    \end{enumerate}
\end{lemma}

\begin{proof}
Recall that the $(2,3,1)$-gYBE is $R_1R_2R_1=R_2R_1R_2$ for $R_1=R\otimes I=\begin{pmatrix}
X\ot I&0\\0& Y\ot I
\end{pmatrix}$ and $R_2=I\otimes R=\begin{pmatrix}
R&0\\0& R
\end{pmatrix}$.
The left-hand side is $$\begin{pmatrix}
X\ot I&0\\0& Y\ot I
\end{pmatrix}
\begin{pmatrix}
R&0\\0& R
\end{pmatrix}
\begin{pmatrix}
X\ot I&0\\0& Y\ot I
\end{pmatrix}
=\begin{pmatrix}
(X\ot I)R(X\ot I)&0\\0& (Y\ot I)R(Y\ot I)
\end{pmatrix},$$
while the right-hand side is $$
\begin{pmatrix}
R&0\\0& R
\end{pmatrix}
\begin{pmatrix}
X\ot I&0\\0& Y\ot I
\end{pmatrix}
\begin{pmatrix}
R&0\\0& R
\end{pmatrix}
 =\begin{pmatrix}
R(X\ot I)R&0\\0& R(Y\ot I)R
\end{pmatrix}.$$
Hence the $(2,3,1)$-gYBE for $R=X\oplus Y$ is equivalent to
    \begin{enumerate}
    \item $ (X\ot I)R(X\ot I)=R(X\ot I)R$,
    \item $(Y\ot I)R(Y\ot I)=R(Y\ot I)R.$
    \end{enumerate}
Substituting $ X\ot I=\frac{1}{\sqrt{2}}\begin{pmatrix}
A\ot I& B\ot I\\C\ot I& D\ot I
\end{pmatrix}$ and $R=\begin{pmatrix}
X&0\\0& Y
\end{pmatrix}$ into the first equation, we get

$$\begin{pmatrix}
A\ot I& B\ot I\\C\ot I& D\ot I
\end{pmatrix}
 \begin{pmatrix}
X&0\\0& Y
\end{pmatrix}
\begin{pmatrix}
A\ot I& B\ot I\\C\ot I& D\ot I
\end{pmatrix}=\sqrt{2}
\begin{pmatrix}
X&0\\0& Y
\end{pmatrix}
\begin{pmatrix}
A\ot I& B\ot I\\C\ot I& D\ot I
\end{pmatrix}
\begin{pmatrix}
X&0\\0& Y
\end{pmatrix}.$$

Comparing the blocks, we obtain the first $4$ equations.  By replacing $A,B,C,D$ with $Y_i,i=1,2,3,4$, respectively, we obtain the
second $4$ equations.

\end{proof}

\begin{lemma}\label{Lemma:YbyX}
Suppose an $8\times 8$ matrix $R=X\oplus Y$ is a solution to the $(2,3,1)$-gYBE such that $X$ is $(2\times 2)$-diagonally unitary.  Then,
    \begin{enumerate}
    \item Equation (1) of Lemma \ref{gYBE:Block} is equivalent to
        \begin{enumerate}
        \item $Y_1=-\bar{\gamma} A(A-\omega D-I)=\begin{pmatrix}\omega &0\\0&\omega \bar{\gamma}(1+\delta \omega -\omega)\end{pmatrix}$,
        \item $Y_2=-\bar{\alpha}\beta\bar{\delta}\bar{\omega}B(A+\omega D-\omega I)=\begin{pmatrix}\beta \bar{\delta}(1-\gamma-\bar{\omega}) &0\\0&-\bar{\alpha}\beta^2\end{pmatrix}$,
        \item $Y_3=\alpha \bar{\beta}\bar{\gamma}B^{\dagger}DA(A+\omega D-\omega I)=\begin{pmatrix}\bar{\beta}(1 +\omega\gamma -{\omega})&0\\0& \alpha\bar{\beta}^2 \delta^2 \omega^2\bar{\gamma}\end{pmatrix}$,
        \item $Y_4=\bar{\delta}\bar{\omega}D(A-\omega D+\omega^2)=\begin{pmatrix}\bar{\delta}\gamma(\omega+\bar{\omega}- \gamma)&0\\0&1-\delta +\omega\end{pmatrix}.$
        \end{enumerate}
    \item Equation (2) of Lemma \ref{gYBE:Block} is equivalent to
        \begin{enumerate}
        \item $(\delta -1)\omega^2+(1+\delta^2-\delta\gamma+\gamma)\omega-2\gamma=0,$
        \item $\delta(2-\bar{\omega}-\gamma)=-\bar{\omega}-\gamma+1+\bar{\omega}\gamma-\gamma^2+\omega\gamma,$
        \item $ (\delta -1)\gamma=\delta^2 -1+ \omega(1-\delta),$
        \item $\delta(2\omega+\bar{\omega}-\gamma)=\bar{\omega}-1+\gamma+\bar{\omega}\gamma+\omega\gamma-\gamma^2.$
        \end{enumerate}
    \item $X$ is a solution to the YBE in dimension $2$ if and only if $X=Y$. 
    \item $Y$ is unitary if and only if
        \begin{enumerate}
        \item $\omega+\bar{\omega}+\gamma+\bar{\gamma}=\omega\gamma+\bar{\omega}\bar{\gamma}+2,$
        \item $\omega+\bar{\omega}+\delta+\bar{\delta}=\omega\delta+\bar{\omega}\bar{\delta}+2,$
        \item $1+\omega+\bar{\omega}+\omega\gamma=\gamma+\bar{\gamma}+\omega^2+\omega\bar{\gamma},$
        \item $2+\omega\delta=\delta+\bar{\delta}+\omega\bar{\delta},$
        \item $\omega+\bar{\omega}+\gamma+\bar{\gamma}+\omega\bar{\gamma}+\bar{\omega}\gamma=4+\omega^2+{\bar{\omega}}^2,$
        \item $\delta+\bar{\delta}+\bar{\omega}\delta+\omega\bar{\delta}=2+\omega+\bar{\omega}.$
        \end{enumerate}
    \item $Y$ is unitary if and only if $\omega, \gamma, \delta$ fall into one of the following three categories:
        \begin{enumerate}
        \item $\omega=\gamma=\pm i, \delta=1$,
        \item $\omega=\delta=\pm i, \gamma=1,$
        \item $\omega=\gamma=\delta=1$.
        \end{enumerate}
    \end{enumerate}
\end{lemma}

\begin{proof}

For (1), by equation (1) of Lemma \ref{gYBE:Block},
$$Y=(B^{\dagger}\otimes I)[\sqrt{2}\cdot X(A\otimes I)X-(A\otimes I)X(A\otimes I)](C^{\dagger}\otimes I).$$
Substituting $A\ot I=\begin{pmatrix}
I &0\\0& \omega I
\end{pmatrix}$, $B^{\dagger}\otimes I=\begin{pmatrix}
\bar{\alpha} I &0\\0& \bar{\beta}I
\end{pmatrix}$,
$C^{\dagger}\otimes I=\begin{pmatrix}
-\bar{\gamma}\alpha I &0\\0& -\beta \bar{\delta} \bar{\omega} I
\end{pmatrix}$, and $X=\frac{1}{\sqrt{2}}\begin{pmatrix}
A&B\\C&D
\end{pmatrix}$
into the equation, we obtain

$Y=\frac{1}{\sqrt{2}}\begin{pmatrix}
-\bar{\gamma}A^2-\bar{\gamma}\omega BC+\bar{\gamma}A & -\bar{\alpha}\bar{\delta}\beta \bar{\omega}AB-\bar{\alpha}\bar{\delta}\beta BD+\bar{\alpha}\bar{\delta}\beta B \\-\bar{\gamma}\alpha \bar{\beta} CA-\bar{\gamma} {\alpha}\bar{\beta} \omega DC+\bar{\gamma}{\alpha}\bar{\beta} \omega C& -\bar{\delta}\bar{\omega} CB-\bar{\delta} D^2+\bar{\delta}\omega D
\end{pmatrix}.$

The formulas follow from replacing $C$ with $-DB^{\dagger}A$ and simplifying.

For (2), we use equation (2) of Lemma \ref{gYBE:Block}.

For (3), $X$ satisfies the YBE in dimension $2$ if and only if $X$ satisfies the first $4$ equations in Lemma \ref{Lemma:YbyX} when $Y$ is replaced by $X$.  Now $(3)$ follows.

For (4), $YY^{\dagger}=I_4$ is the same as
    \begin{enumerate}
    \item $Y_1Y_1^{\dagger}+Y_2Y_2^{\dagger}=2I,$
    \item $Y_1Y_3^{\dagger}+Y_2Y_4^{\dagger}=0,$
    \item $Y_3Y_1^{\dagger}+Y_4Y_2^{\dagger}=0,$
    \item $Y_3Y_3^{\dagger}+Y_4Y_4^{\dagger}=2I.$
    \end{enumerate}
Substituting $Y_i,i=1,2,3,4$, from above gives the equations.

For (5), by (d) above, we have $2-\delta-\bar{\delta}=\omega(\bar{\delta}-\delta)$.  It follows that $\omega(\bar{\delta}-\delta)$ is a real number. Let $\omega=x+iy, \delta=a+ib$.  Since $(x+iy)(-2ib)$ is real, its imaginary part is $-2xb=0$.  If $x=0$, then $\omega=\pm i$.  If $b=0$, then $2-\delta-\bar{\delta}=0$; hence $\delta=1$.

Case $1$:  $\omega=i$

The equations for $Y$ to be unitary become
    \begin{enumerate}
    \item $(1-i)\gamma+(1+i)\bar{\gamma}=2,$
    \item $(1-i)\delta+(1+i)\bar{\delta}=2.$
    \end{enumerate}
Thus, the real part of $(1-i)\gamma$ is $1$.  When we let $\gamma=c+di$, $c-d=1$.  Since $c^2+d^2=1$, it follows that either $c=1,d=0$ or $c=0,d=1$; i.e., $\gamma=i$ or $\gamma=1$.  Similarly, $\delta=i$ or $\delta=1$.

If $\delta=i$, then equation (c) of (2) implies that $\gamma=\delta+1-i=1.$  Therefore, $\omega=\delta=i, \gamma=1$.

If $\delta=1$, substituting into (a) of (2), we get $\omega=\gamma=i$.  Therefore, $\omega=\gamma=i, \delta=1$.

Case $2$: $\omega=-i$

This case is the complex conjugate of Case $1$.  Therefore, $\omega=\delta=-i, \gamma=1$ or $\omega=\gamma=-i, \delta=1$.

Case $3$: $\delta=1$

By $(a)$ of (2), we have $\omega=\gamma$.  Then the equations for $Y$ to be unitary become
$2(\omega+\bar{\omega})=\omega^2+\bar{\omega}^2+2.$
Hence $\omega+\bar{\omega}=0$ or $2$.  It follows that $\omega=\pm i$ or $\omega=1$.  Therefore, $\omega=\gamma=\pm i, \delta=1$ or $\omega=\gamma=\delta=1$.

\end{proof}

\begin{lemma}\label{Lemma:reduction}
Let $\tilde{R}=\tilde{X}\oplus \tilde{Y}$, with $\tilde{X}$ and $\tilde{Y}$ being the conjugates of $X$ and $Y$ as follows:
$\tilde{X}=\begin{pmatrix}
I &0\\0&B
\end{pmatrix}
X
\begin{pmatrix}
I &0\\0&B^{\dagger}
\end{pmatrix}$ and
$\tilde{Y}=\begin{pmatrix}
I &0\\0&\bar{\alpha}\beta B
\end{pmatrix}
Y
\begin{pmatrix}
I &0\\0& \alpha \bar{\beta} B^{\dagger}
\end{pmatrix}$.  Let $\tilde{A}, \tilde{B}, \tilde{C}, \tilde{D}, \tilde{Y_i},i=1,2,3,4$, denote the
corresponding matrices in $\tilde{R}$.

If $R$ is a solution to the $(2,3,1)$-gYBE and $X$ is $(2\times 2)$-diagonally unitary, then
$\tilde{R}$ is also a solution to the $(2,3,1)$-gYBE with $\tilde{B}$ being the identity matrix $I$.  Conversely, if $\tilde{R}$ is
a solution to the $(2,3,1)$-gYBE such that $\tilde{B}=I$ and $\tilde{X}$ is $(2\times 2)$-diagonally unitary, and $B=\begin{pmatrix}
\alpha &0\\0&\beta
\end{pmatrix}$ is an arbitrary unitary diagonal matrix, then $R=X\oplus Y$, which is similarly obtain from $\tilde{R}$ and $B$, is also a solution.
\end{lemma}

\begin{proof}
Direct computation shows that
$$\tilde{X}=\frac{1}{\sqrt{2}}\begin{pmatrix}
A &I\\
-DA& D
\end{pmatrix},
\tilde{Y}=\frac{1}{\sqrt{2}}\begin{pmatrix}
Y_1 &\alpha \bar{\beta} B^{\dagger} Y_2\\
\bar{\alpha} \beta B Y_3& Y_4
\end{pmatrix}.$$  Therefore,
$$\tilde{A}=A, \tilde{B}=I, \tilde{C}=-DA, \tilde{D}=D,$$ $$\tilde{Y_1}=Y_2, \tilde{Y_2}=\alpha \bar{\beta} B^{\dagger} Y_2,
\tilde{Y_3}=\bar{\alpha} \beta B Y_3, \tilde{Y_4}=Y_4.$$

Replacing all matrices in Lemma \ref{gYBE:Block} by the $\tilde{R}$ matrices using the above identities,
we find that $\tilde{R}$ is a solution to the $(2,3,1)$-gYBE if and only if
    \begin{enumerate}
    \item $(A\otimes I)\tilde{X} (A\otimes I)+\tilde{Y} ((-DA)\otimes I)=\sqrt{2}\cdot \tilde{X}(A\otimes I)\tilde{X}$
    \item $(A\otimes I)\tilde{X}+\tilde{Y}(D\otimes I)=\sqrt{2}\cdot \tilde{X}\tilde{Y}$
    \item $((-DA)\otimes I)\tilde{X} (A\otimes I)+(D\otimes I) \tilde{Y}((-DA)\otimes I)=\sqrt{2}\cdot \tilde{Y}((-DA)\otimes I)\tilde{X}$
    \item $((-DA)\otimes I)\tilde{X}+(D\otimes I) \tilde{Y}(D\otimes I)=\sqrt{2}\cdot \tilde{Y}(D\otimes I)\tilde{Y}$
    \item $(Y_1\otimes I)\tilde{X} (Y_1\otimes I)+((\alpha \bar{\beta} B^{\dagger} Y_2)\otimes I) \tilde{Y}((\bar{\alpha} \beta B Y_3)\otimes I)=\sqrt{2}\cdot \tilde{X}(Y_1\otimes I)\tilde{X}$
    \item $(Y_1\otimes I)\tilde{X} ((\alpha \bar{\beta} B^{\dagger} Y_2)\otimes I)+((\alpha \bar{\beta} B^{\dagger} Y_2)\otimes I) \tilde{Y}(Y_4\otimes I)=\sqrt{2}\cdot \tilde{X}((\alpha \bar{\beta} B^{\dagger} Y_2)\otimes I)\tilde{Y}$
    \item $((\bar{\alpha} \beta B Y_3)\otimes I)\tilde{X} (Y_1\otimes I)+(Y_4\otimes I) \tilde{Y}((\bar{\alpha} \beta B Y_3)\otimes I)=\sqrt{2}\cdot \tilde{Y}((\bar{\alpha} \beta B Y_3)\otimes I)\tilde{X}$
    \item $((\bar{\alpha} \beta B Y_3)\otimes I)\tilde{X} ((\alpha \bar{\beta} B^{\dagger} Y_2)\otimes I)+(Y_4\otimes I) \tilde{Y}(Y_4\otimes I)=\sqrt{2}\cdot \tilde{Y}(Y_4\otimes I)\tilde{Y}$
    \end{enumerate}
It suffices to show that the above $8$ equations are the same as the $8$ equations of Lemma \ref{gYBE:Block}.

The $8$ cases are similar, so we show only the first one in detail.  Substituting $\tilde{X}$ and $\tilde{Y}$ into the first equation, we have
$$(A\otimes I)
\begin{pmatrix}
I &0\\0&B
\end{pmatrix}
X
\begin{pmatrix}
I &0\\0&B^{\dagger}
\end{pmatrix}
 (A\otimes I)+
 \begin{pmatrix}
I &0\\0&\bar{\alpha}\beta B
\end{pmatrix}
Y
\begin{pmatrix}
I &0\\0& \alpha \bar{\beta} B^{\dagger}
\end{pmatrix}
 ((-DA)\otimes I)$$
 $$=\sqrt{2}\cdot
 \begin{pmatrix}
I &0\\0&B
\end{pmatrix}
X
\begin{pmatrix}
I &0\\0&B^{\dagger}
\end{pmatrix}
 (A\otimes I)
 \begin{pmatrix}
I &0\\0&B
\end{pmatrix}
X
\begin{pmatrix}
I &0\\0&B^{\dagger}
\end{pmatrix}.$$

Multiplying the equation from the left by $\begin{pmatrix}
I &0\\0&B^{\dagger}
\end{pmatrix}$ and from the right by $\begin{pmatrix}
I &0\\0&B
\end{pmatrix}$ and noticing that $\begin{pmatrix}
I &0\\0&B^{\dagger}
\end{pmatrix}$, $\begin{pmatrix}
I &0\\0&B
\end{pmatrix}$, and $A\ot I$ all commute with one another, we obtain
$$(A\ot I) X(A\ot I)+
\begin{pmatrix}
I &0\\0&\bar{\alpha}\beta I
\end{pmatrix}
Y
\begin{pmatrix}
I &0\\0&{\alpha}\bar{\beta} I
\end{pmatrix}
(-DA\ot I)=\sqrt{2} X(A\ot I)X.$$

Since $\begin{pmatrix}
I &0\\0&\bar{\alpha}\beta I
\end{pmatrix}=\bar{\alpha}
\begin{pmatrix}
\alpha I &0\\0&{\beta} I
\end{pmatrix} =\bar{\alpha} (B\ot I)$ and
$\begin{pmatrix}
I &0\\0&{\alpha}\bar{\beta} I
\end{pmatrix}={\alpha}
\begin{pmatrix}
\bar{\alpha} I &0\\0&\bar{\beta} I
\end{pmatrix} ={\alpha} (B^{\dagger} \ot I)$, we have
$$(A\ot I)X(A\ot I)+(B\ot I)Y(-B^{\dagger}DA\ot I)=\sqrt{2}X(A\ot I)X,$$
which is equation (1).

\end{proof}

In particular, Lemma \ref{Lemma:reduction} shows that any  solution $R=X\oplus Y$ with $X$ $(2\times 2)$-diagonally unitary can be recovered from a solution with $B=I$ and the values $\alpha$ and $\beta$. In the unitary case we have the following:

\begin{lemma}\label{Lemma:3solutions}
The matrices $R=X\oplus Y$ for $\alpha=\beta=\delta=1, \omega=\gamma=i$; $\alpha=\beta=\gamma=1, \omega=\delta=i$; and $\alpha=\beta=\omega=\gamma=\delta=1$ are the following three unitary solutions, respectively, to the $(2,3,1)$-gYBE:
    \begin{enumerate}
    \item $$R=\frac{1}{\sqrt{2}} \begin{pmatrix}1&0&1 &0\\0 & i&0& 1 \\-i &0&i&0\\0&-i  &0& 1\end{pmatrix}\oplus \frac{1}{\sqrt{2}} \begin{pmatrix}i&0&1 &0\\0 & 1 &0&-1\\-i &0&1&0\\0& i &0& i\end{pmatrix}.$$
    \item $$R=\frac{1}{\sqrt{2}} \begin{pmatrix}1&0&1 &0\\0 & i&0& 1 \\-1 &0&1&0\\0&1&0& i\end{pmatrix}\oplus \frac{1}{\sqrt{2}} \begin{pmatrix}i&0&1 &0\\0 & 1 &0&-1\\1&0&i&0\\0& 1 &0& 1\end{pmatrix}.$$
    \item $$R=\frac{1}{\sqrt{2}} \begin{pmatrix}1&0&1 &0\\0 & 1&0& 1 \\-1&0&1&0\\0&-1&0& 1\end{pmatrix}\oplus \frac{1}{\sqrt{2}} \begin{pmatrix}1&0&-1&0\\0 & 1 &0&-1\\1 &0&1&0\\0& 1 &0& 1\end{pmatrix}.$$
    \end{enumerate}
\end{lemma}

We have not included the solutions for $\alpha=\beta=\delta=1, \omega=\gamma=i$, and $\alpha=\beta=\gamma=1, \omega=\delta=i$, which are merely the conjugates of solutions (1) and (2), respectively.

\begin{proof}
It suffices to show that the three matrices satisfy the equations $(1)-(8)$ in Lemma \ref{gYBE:Block}.

The verification is a long matrix algebra computation.  This can be done using a software package such as Maple.  To do the computation algebraically, we notice that the verification is equivalent to checking 8 sets of matrix equations. We have done both the Maple verification and the algebraic computation.
\end{proof}

\begin{lemma}\label{Lemma:generalsolutions}
The following are true:
    \begin{enumerate}
    \item For each $(2,3,1)$-$R$-matrix $R=X\oplus Y$ in Lemma \ref{Lemma:3solutions}, $X$ and $Y$ are conjugate to each other.  Their common eigenvalues are $\{e^{-\frac{\pi i}{12}},e^{-\frac{\pi i}{12}},e^{\frac{7\pi i}{12}},e^{\frac{7\pi i}{12}}\}$ for the first family, $\{e^{-\frac{\pi i}{4}},-e^{-\frac{\pi i}{4}},e^{\frac{\pi i}{4}},e^{\frac{\pi i}{4}}\}$ for the second family, and $\{e^{-\frac{\pi i}{4}},e^{-\frac{\pi i}{4}},e^{\frac{\pi i}{4}},e^{\frac{\pi i}{4}}\}$ for the third family.  It follows that any two different $(2,3,1)$-$R$-matrices above are not conjugate to each other.
    \item For any $\alpha, \beta$ such that $|\alpha|=1, |\beta|=1$, the matrix $R(\alpha,\beta)$ in each family below is a unitary solution to the $(2,3,1)$-gYBE:
        \begin{enumerate}
        \item $$R(\alpha, \beta)=\frac{1}{\sqrt{2}} \begin{pmatrix}1&0&\alpha &0\\0 & i&0& \beta \\-i\bar{\alpha} &0&i&0\\0&-i \bar{\beta} &0& 1\end{pmatrix}\oplus \frac{1}{\sqrt{2}} \begin{pmatrix}i&0&\beta &0\\0 & 1 &0&-\bar{\alpha} \beta^2\\-i\bar{\beta} &0&1&0\\0& i \alpha {\bar{\beta}}^2 &0& i\end{pmatrix}.$$
        \item $$R(\alpha, \beta)=\frac{1}{\sqrt{2}} \begin{pmatrix}1&0&\alpha &0\\0 & i&0& \beta \\-\bar{\alpha} &0&1&0\\0&\bar{\beta}&0& i\end{pmatrix}\oplus \frac{1}{\sqrt{2}} \begin{pmatrix}i&0&\beta &0\\0 & 1 &0&-\bar{\alpha} \beta^2\\\bar{\beta}&0&i&0\\0& \alpha {\bar{\beta}}^2 &0& 1\end{pmatrix}.$$
        \item $$R(\alpha, \beta)=\frac{1}{\sqrt{2}} \begin{pmatrix}1&0&\alpha &0\\0 & 1&0& \beta \\-\bar{\alpha}&0&1&0\\0&-\bar{\beta}&0& 1\end{pmatrix}\oplus \frac{1}{\sqrt{2}} \begin{pmatrix}1&0&-\beta&0\\0 & 1 &0&-\bar{\alpha} \beta^2\\\bar{\beta} &0&1&0\\0& \alpha {\bar{\beta}}^2 &0& 1\end{pmatrix}.$$
        \end{enumerate}
    \item In each family above, $R(\alpha,\beta)$ is locally equivalent to $R(\alpha',\beta')$ if and only if $\frac{\beta}{\alpha}=\frac{\beta'}{\alpha'}$.
    \end{enumerate}
\end{lemma}

\begin{proof}
For (1), $Y=P^{\dagger}XP$ by $P$ equal to
$\begin{pmatrix}
    0&i\sigma_z\\I&0
\end{pmatrix}$,
$\begin{pmatrix}
    0&\sigma_x\\\sigma_x & 0
\end{pmatrix}$, and
$\begin{pmatrix}
    0&I\\I&0
\end{pmatrix}$,
where
$\sigma_x=\begin{pmatrix}
    0&1\\1&0
\end{pmatrix}$ and
$\sigma_z=\begin{pmatrix}
    1&0\\0&-1
\end{pmatrix}$ are the Pauli matrices.

The characteristic polynomials of $X$ for the three cases are $(\lambda^2-e^{\frac{\pi i}{4}}\lambda+i)^2$, $((\lambda-\frac{i}{\sqrt{2}})^2-\frac{1}{2})((\lambda-\frac{1}{\sqrt{2}})^2+\frac{1}{2})$,
and $((\lambda-\frac{1}{\sqrt{2}})^2+\frac{1}{2})^2$.

The solutions in (2) follow from Lemma \ref{Lemma:reduction}.

For (3), suppose
$Q=\begin{pmatrix}
    a&b\\c&d
\end{pmatrix}$ is an invertible $2\times 2$ matrix such that
$$ (Q\otimes Q\otimes Q)R(\alpha',\beta')=R(\alpha',\beta')(Q\otimes Q\otimes Q).$$
This equation is the same as
    \begin{enumerate}
    \item $a(Q\ot Q)X(\alpha',\beta')=a X(\alpha,\beta)(Q\ot Q),$
    \item $b(Q\ot Q)Y(\alpha',\beta')=b X(\alpha,\beta)(Q\ot Q),$
    \item $c(Q\ot Q)X(\alpha',\beta')=c Y(\alpha,\beta)(Q\ot Q),$
    \item $d(Q\ot Q)Y(\alpha',\beta')=d Y(\alpha,\beta)(Q\ot Q).$
    \end{enumerate}
If neither $a$ nor $b$ is $0$, then (1) and (2) imply $X(\alpha',\beta')=Y(\alpha',\beta')$.  This contradiction implies that
one of $a,b$ is $0$.  Since $Q$ is invertible, therefore, exactly one of $a,b$ is $0$.  Similarly, using (3) and (4), we know that
exactly one of $c,d$ is $0$.  Since $Q$ is invertible, there are exactly two cases: $b=c=0$ and $a\neq 0, d\neq 0$, or $a=d=0$ and $b\neq 0, c\neq 0$.

If $Q=\begin{pmatrix}
    a&0\\0&d
\end{pmatrix}$, then
    \begin{enumerate}
    \item $aQA(\alpha',\beta')=a A(\alpha,\beta)Q,$
    \item $aQ B(\alpha',\beta')=d B(\alpha,\beta)Q,$
    \item $dQ C(\alpha',\beta')=a C(\alpha,\beta)Q,$
    \item $dQD(\alpha',\beta')=d D(\alpha,\beta)Q.$
    \end{enumerate}
Noting that $Q,A,B,C,D$ all commute with one another, we have $A(\alpha',\beta')=A(\alpha,\beta), B(\alpha',\beta')=\frac{d}{a} B(\alpha,\beta)$ and $C(\alpha',\beta')=\frac{c}{d} C(\alpha,\beta),
D(\alpha',\beta')=D(\alpha,\beta)$.  Therefore, $R(\alpha,\beta)$ is locally equivalent to $R(\alpha',\beta')$ by a diagonal $Q$ if and only if $\frac{\beta}{\alpha}=\frac{\beta'}{\alpha'}$.

If $Q=\begin{pmatrix}
    0&b\\c&0
\end{pmatrix}$, then
    \begin{enumerate}
    \item $bQ Y_3(\alpha',\beta')=c B(\alpha,\beta)Q,$
    \item $Q Y_4(\alpha',\beta')=A(\alpha,\beta)Q,$
    \item $Q Y_1(\alpha',\beta')=D(\alpha,\beta)Q,$
    \item $cQY_2(\alpha',\beta')=b C(\alpha,\beta)Q.$
    \end{enumerate}

For the first family of solutions, this is impossible.  For the second and third families, the equations lead to the same conclusions as for the diagonal $Q$.

\end{proof}

{\bf Completion of the Proof of Theorem \ref{Thm:Rowell}}:
If $R$ is a $(2,3,1)$-$R$-matrix, by multiplying by a scalar, we may assume that the $(1,1)$-entry is $\frac{1}{\sqrt{2}}$.  Therefore,
it is one of the $R(\alpha,\beta)$ in Lemma \ref{Lemma:generalsolutions}.  By Lemma \ref{Lemma:generalsolutions}, any two R-matrices in the same
family are locally conjugate if and only if $\frac{\beta}{\alpha}$ are the same.  Setting  $e^{i \theta}=\frac{\beta}{\alpha}$, we know that all local conjugation classes up to a scalar are represented by  $0\leq \theta < 2\pi$.  Complex conjugation reduces the equivalence classes to $0\leq \theta \leq \pi$.

\qed

\section{Braiding Quantum Gates}\label{quantumgates}

In quantum information science, information is encoded in quantum states and processed by unitary evolutions.  In topological quantum computation, unitary evolutions are effected by braiding anyons, $2$-dimensional quasi-particles \cite{Nayak}.  Braiding anyons gives rise to unitary representations of the braid group.  All of our solutions to the $(2,3,1)$-gYBE lead to unitary representations of the braid group and therefore can be used in quantum information science.

\subsection{Far Commutativity}

Recall that the $n$-strand braid group $B_n$ has a presentation with $n-1$ generators $\sigma_1,\sigma_2, ..., \sigma_{n-1}$, as well as two defining relations:
\begin{enumerate}
    \item $\sigma_i\sigma_{i+1}\sigma_i=\sigma_{i+1}\sigma_i\sigma_{i+1}$ for $1\leq i \leq n-2$
    \item $\sigma_i\sigma_j=\sigma_j\sigma_i$ for $|i-j|\geq 2$.
\end{enumerate}
The first is commonly called the braid relation and the second far commutativity.

A unitary $R$-matrix naturally gives rise to a unitary representation of the $n$-strand braid group $B_n$ on $V^{\ot n}$ as follows.  Set
$$R_i=I^{\ot (i-1)} \otimes R\otimes I^{\ot (n-i-1)}.$$
Then the assignment $\rho_R(\sigma_i)=R_i$ is a unitary matrix representation of $B_n$.  The braid relation
$$\rho_R(\sigma_i)\rho_R(\sigma_{i+1})\rho_R(\sigma_i)=\rho_R(\sigma_{i+1})\rho_R(\sigma_{i})\rho_R(\sigma_{i+1})$$
follows exactly from the YBE for $R$.
The far commutativity relation $\rho_R(\sigma_{i})\rho_R(\sigma_j)=\rho_R(\sigma_{j})\rho_R(\sigma_i)$ if $|i-j|>1$ holds because $\rho_R(\sigma_{i})$ and $\rho_R(\sigma_{j})$ act nontrivially on disjoint tensor factors of $V^{\ot n}$.

If $R$ is a solution to a $(d,m,l)$-gYBE, we can perform a similar assignment: for $B_n$, set
$$R_{\sigma_i}=I^{\ot{l(i-1)}} \otimes R \otimes I^{\ot{l(n-i-1)}}.$$
Then the braid relation holds because of the gYBE for $R$.  But the far commutativity relation is not necessarily satisfied when $|i-j|>1$.  However, if $2l \geq m$, then it is still true that $\rho_R(\sigma_{i})$ and $\rho_R(\sigma_{j})$ act nontrivially on disjoint tensor factors of $V^{\ot(m+(n-2)l)}$ and that the far commutativity condition is satisfied.  It follows that the $X$-shape $(2,3,2)$-$R$-matrix $R_X$ leads to unitary braid group representations.  Otherwise, even when $|i-j|>1$,  the two matrices $\rho_R(\sigma_i)$ and $\rho_R(\sigma_j)$ might not commute with each other.

\begin{defn}
Suppose $R$ is a $(d,m,l)$-$R$-matrix.  Then $R$ satisfies the far commutativity condition if
$$R_{\sigma_1}R_{\sigma_j}=R_{\sigma_j}R_{\sigma_1}$$ in the braid group $B_{(j-1)l+2}$ for every $j$ such that $j>2$ and $(j-1)l<m$.
\end{defn}

If a $(d,m,l)$-$R$-matrix satisfies the far commutativity condition, then it yields a representation of each $n$-strand braid group $B_n$ on $V^{\ot (m+(n-2)l)}$.  In particular, any $(d,m,l)$-$R$-matrix with $2l\geq m$ leads to a representation of the braid group.

\begin{prop}
Let $R=X\oplus Y$ be a $(2,3,1)$-$R$-matrix such that $X=\frac{1}{\sqrt{2}}\begin{pmatrix}
A&B\\C&D
\end{pmatrix}$ and $Y=\frac{1}{\sqrt{2}}\begin{pmatrix}
Y_1&Y_2\\Y_3&Y_4
\end{pmatrix}$ for $2\times 2$ matrices $A,B,C,D, Y_i,i=1,2,3,4$.
    \begin{enumerate}
    \item If $A,B,C,D, Y_i, i=1,2,3,4$, are all diagonal, then $R_{\sigma_1}R_{\sigma_3}=R_{\sigma_3}R_{\sigma_1}$.
    \item If one of $A,B,C,D, Y_i, i=1,2,3,4$, is not diagonal, then $R_{\sigma_1}R_{\sigma_3}=R_{\sigma_3}R_{\sigma_1}$ implies $X=Y$.
    \end{enumerate}
\end{prop}

\begin{proof}
The far commutativity condition is equivalent to
    \begin{enumerate}
    \item $(X\ot I_4)(I\ot R)=(I\ot R)(X\ot I_4),$
    \item $(Y\ot I_4)(I\ot R)=(I\ot R)(Y\ot I_4).$
    \end{enumerate}
These equations are the same as
    \begin{enumerate}
    \item $(A\ot I_4)R=R(A\ot I_4),$
    \item $(B\ot I_4)R=R(B\ot I_4),$
    \item $(C\ot I_4)R=R(C\ot I_4),$
    \item $(D\ot I_4)R=R(D\ot I_4),$
    \item $(Y_i\ot I_4)R=R(Y_i\ot I_4), i=1,2,3,4.$
    \end{enumerate}
Let $A=\begin{pmatrix}
a_{11}&a_{12}\\a_{21}&a_{22}
\end{pmatrix}.$  Then the first equation is
    \begin{enumerate}
    \item $a_{11}X=a_{11}X$
    \item $a_{12}Y=a_{12}X$
    \item $a_{21}X=a_{21}Y$
    \item $a_{22}Y=a_{22}Y.$
    \end{enumerate}
Now the proposition follows.

\end{proof}

\begin{cor}
All $(2,3,1)$-$R$-matrices in Thm.\ \ref{Thm:Rowell} satisfy the far commutativity condition and thus yield unitary representations of the braid group $B_n$ on $(\mathbb{C}^2)^{\ot (n+1)}$.
\end{cor}

\subsection{$R$-braiding Quantum Gates}

\begin{defn}
Let $R$ be a unitary $(2,3,1)$-$R$-matrix that leads to a representation $\rho_R$ of the braid group.  A unitary matrix $U_i\in \mathbb{U}(2^{n+1})$ is an $R$-braiding gate if $U_i=\lambda \cdot \rho_R(\sigma_i)$ for some braid generator $\sigma_i$ and non-zero scalar $\lambda$.

A unitary matrix $U$ is an $R$-braiding quantum circuit if there exists a braid $b$ such that $U=\lambda \cdot \rho_R(b)$ for some non-zero scalar $\lambda$.
\end{defn}

Quantum circuits are compositions of quantum gates and are used to perform computations.  It follows that each $(2,3,1)$-$R$ matrix in Thm.\ \ref{Thm:Rowell} leads to $(n+1)$-qubit braiding quantum circuits from its associated representation of the $n$-strand braid group $B_n$ for all $n\geq 1$.

\section{Open Questions}\label{questions}

Generalized Yang-Baxter equations provide a new method for finding representations of the braid group.  These representations have applications in a variety of fields; in particular, the resulting braiding quantum circuits are actively studied in quantum information science \cite{Nayak} \cite{KL} \cite{RZWG}.  In this section, we list a few of the many open questions and directions for further research.
    \begin{enumerate}
    \item How can we find all unitary $(2,3,1)$-$R$-matrices $R=X\oplus Y$ such that $X, Y$ are $2\times 2$ diagonal and unitary?  Note that their $2\times 2$  blocks $A,B,C,D,Y_i,i=1,2,3,4$, are not necessarily unitary.
    \item How do we find all unitary $(2,3,1)$- or $(2,3,2)$-$R$-matrices with the zero entries in the same positions as in the solution $R_X$?
    \item Which quantum circuits can be realized by the braiding quantum circuits resulting from our $(2,3,1)$-$R$-matrices?  Mathematically, the question is to find the images of the afforded braid group representations.
    \item Which entangled states can be generated from product states using our braiding quantum circuits?
    \item For the three families of solutions in Thm\ \ref{Thm:Rowell}, do any two $(2,3,1)$-$R$-matrices $R$ and $R'$ from the same family lead to equivalent braid group representations?  To determine whether the representations from $R$ and $R'$ are equivalent, we need to find out whether there exists a single matrix $P$ such that $$P^{-1} (R\otimes I) P=R'\otimes I, \;\;\; P^{-1} (I\otimes R) P=I\otimes R'.$$ We hope to settle this question in the future.  Note that the answer has no bearing on the application to braiding quantum gates.
    \item Are our new representations of the braid group equivalent to known braid group representations obtained from other methods such as quantum groups?

\end{enumerate}

Ultimately we would like to know if there are real physical systems that would realize our braid group representations.  Such physical systems would be quantum computers, powerful tools for exploring the quantum world and for bringing us new technologies of great benefit to society.

\end{document}